\documentclass{amsart}
\usepackage[english]{babel}
\usepackage{latexsym,amsfonts,amsmath,graphics,enumerate,amsthm,amssymb}
\usepackage{hyperref}

\usepackage{color}
\usepackage{graphicx}

\newcommand\closure{\operatorname{cl}}
\newcommand\intersection{\cap}

\newcommand\hil{\operatorname{Hil}}
\newcommand\hildist{\hil}
\newcommand\funk{\operatorname{Funk}}
\newcommand\after{\circ}
\newcommand{\myexp}{\mathop{}\mathopen{}\mathrm{\exp}}
\newcommand\union{\cup}
\newcommand\R{\mathbb{R}}
\newcommand{\Rplus}{\mathbb{R}_+}

\newcommand\N{\mathbb{N}}
\newcommand\coll{\operatorname{Coll}}
\newcommand\isom{\operatorname{Isom}}

\newcommand\revboundary{\mathcal{B}_\text{RF}}
\newcommand\detourrev{H_\text{RF}}
\newcommand\deltarev{\delta_\text{RF}}
\newcommand\deltafunk{\delta_\text{F}}
\newcommand\deltahil{\delta_\text{H}}
\newcommand\hrev{r}
\newcommand\rev{\operatorname{RFunk}}

\newcommand\gauge[3]{M({#2}/{#3};{#1})}
\newcommand\M[3]{M({#1}/{#2};C)}

\newcommand\interior{\operatorname{int}}
\newcommand\PGL{\operatorname{PGL}}

\theoremstyle{definition} 

\newtheorem{definition}{Definition}[section]

\theoremstyle{plain}      

\newtheorem{proposition}[definition]{Proposition}
\newtheorem{theorem}[definition]{Theorem}
\newtheorem{corollary}[definition]{Corollary}
\newtheorem{lemma}[definition]{Lemma}

\begin{document}

\title[Horofunction boundary and isometry group of Hilbert geometry]
{The horofunction boundary and isometry group of the Hilbert geometry}
\author{Cormac Walsh
}
\thanks{Work partially supported by the joint RFBR-CNRS grant number
05-01-02807}

\address{INRIA \& CMAP, \'{E}cole Polytechnique\\ 
91128 Palaiseau \\
France.
}
\email{cormac.walsh@inria.fr}

\begin{abstract}
The horofunction boundary is a means of compactifying metric spaces that was
introduced by Gromov in the 1970s.
We describe explicitly the horofunction boundary of the Hilbert
geometry, and sketch how it may be used to study the isometry group of this
space.
\end{abstract}

\maketitle




\section{Introduction}

Building on the work of Busemann, Gromov defined
in~\cite{gromov:hyperbolicmanifolds}
a certain compactification of a metric space. Recall that the Busemann
function\index{Busemann function} of a geodesic ray $\gamma$ is the
limiting function
\begin{align*}
B_\gamma(\cdot) := \lim_{t\to\infty} d(\cdot,\gamma(t))-d(b,\gamma(t)),
\end{align*}
where $b$ is an arbitrary base-point.
Gromov's idea was to generalise this notion by
considering all possible limits of $d(\cdot,z)-d(b,z)$ as $z$ heads off to
infinity in the metric space, not necessarily along a geodesic ray.
Such a limit is called a horofunction,\index{horofunction}
and may be thought of as a
``point at infinity'' of the metric space. By adjoining the horofunctions
to the metric space, under mild conditions, one compactifies it.
The set of horofunctions thus forms a ``boundary at infinity'' of the
metric space.

The attractiveness of this boundary arises in part from its generality:
whereas the definition of other boundaries, such as the hyperbolic boundary
of a Gromov hyperbolic
space, or the ideal boundary of a CAT(0) space, require the metric space to
have certain properties, the horofunction boundary exists for any metric space.

The horofunction boundary is often closely related to these other boundaries
however. For example, in the CAT(0) case it is actually the
same~\cite{ballmann:spaces}, and in the Gromov hyperbolic case, the hyperbolic
boundary can be recovered from the horofunction boundary by quotienting out
a certain equivalence relation~\cite{coornaert_papadopoulos_horofunctions,
winweb_hyperbolic,storm_barycenter}.

The horofunction boundary has found diverse applications.
Rieffel~\cite{rieffel_group} has used it in his work on quantum metric spaces.
It also appears to be the right general setting for Patterson--Sullivan
measures~\cite{burger_mozes_commensurators}.

In this chapter, we will consider the horofunction boundary of the
Hilbert\index{Hilbert metric}\index{metric!Hilbert} geometry.
Let $D$ be a bounded open convex subset of a
finite-dimensional real linear space. Given a pair of distinct points $x$
and $y$ in $D$, let $w$ and $z$ be points in the usual boundary of $D$
such that $w$, $x$, $y$, and $z$ lie in this order along a straight line;
see Figure~\ref{fig:hilbert_def}.
The Hilbert distance between $x$ and $y$ is defined to be the logarithm of the
cross ratio of these four points:
\begin{align*}
\hil(x,y) := \log\frac{|zx|}{|zy|}\frac{|wy|}{|wx|},
\end{align*}
where $|\cdot|$ denotes any norm on the linear space.
(Be aware that some authors use a different convention here
and divide by a factor of $2$.)
\begin{figure}
\centering
\input{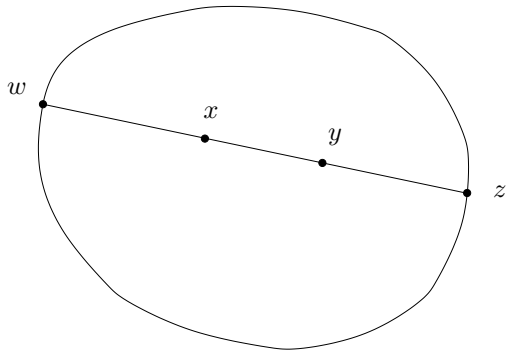}
\caption{Definition of the Hilbert distance.}
\label{fig:hilbert_def}
\end{figure}

The horofunction boundary of this metric was first studied by
Karlsson, Metz, and Noskov~\cite{karl_metz_nosk_horoballs}, who determined
it in the case where $D$ is an open simplex.
The general case was considered in~\cite{walsh_hilbert}. In that paper,
use was made of the fact that the Hilbert metric is the symmetrisation
of the Funk metric:\index{Funk metric}\index{metric!Funk}
$\hil(x,y) = \funk(x,y) + \rev(x,y),$
where
\begin{align*}
\funk(x,y) &:= \log\frac{|zx|}{|zy|},
\qquad\text{and} \\
\rev(x,y) &:= \funk(y,x) = \log\frac{|wy|}{|wx|}.
\end{align*}
The horofunction boundaries of the Funk and
reverse-Funk\index{reverse-Funk metric}\index{metric!reverse-Funk}
geometries were studied separately, and the results were combined
to find the horofunction boundary of the Hilbert geometry.

It turns out that the horofunction boundary of the reverse-Funk geometry
is just the usual boundary of the convex domain $D$. That of the Funk geometry,
however, is more closely related to the usual boundary of the polar body
$D^\circ$ of the domain.
We describe these boundaries in detail in
Section~\ref{sec:hilbert_horoboundary},
and say how they combine to form the Hilbert-geometry horofunction boundary.

Observe that, since every collineation preserves cross-ratios, any
collineation that leaves the domain $D$ invariant is an isometry of the
Hilbert metric on~$D$. There are examples known,
however,~\cite{delaharpe, molnar_thompson_isometries, lemmens_walsh_polyhedral, 
bosche} where there exist other isometries as well, namely, cross-sections
of non-Lorentzian symmetric cones. For these domains, the projective action
of Vinberg's $*$-map, or equivalently, the inverse map in the associated
Jordan algebra, is an isometry but not a collineation, and the isometry group
is generated by this map and the collineations.
It was proved in~\cite{walsh_gauge} that there are no other
examples.

The horofunction boundary is useful for studying isometry groups.
This is because such groups have a natural action by homeomorphisms on the
boundary, and often this action is easier to understand than the action on the
space itself. The horofunction boundary was used in this way
in~\cite{walsh_stretch} to determine the group of isometries of Teichm\"uller
space with Thurston's metric.
It was also used in~\cite{lemmens_walsh_polyhedral}
to determine the isometry group of the Hilbert geometry when the domain
$D$ is polyhedral.
In Section~\ref{sec:hilbert_isometries}, we discuss in more detail the isometry
group of the Hilbert geometry and how the horofunction boundary may be used
to study it.

Before that, however, we recall in the next section the definition of the
horofunction boundary and some of its properties.
To visualize horofunctions, it is often useful to consider their
horoballs,\index{horoball} that is, their sublevel sets.
Intuitively, a horoball looks like a large ball of the metric whose center
is at infinity.
We will make this precise by showing that a sequence of points converges to a
horofunction if and only if sequences of closed balls of the right size
with those points
as centers converge to horoballs of the horofunction, in a certain topology.

\section{The horofunction boundary}
\label{sec:horoboundary}

\newcommand\iso{f}
\newcommand\mapclass{h}
\newcommand\distfn{\psi}
\newcommand\dist{d}
\newcommand\symdist{d_{\text{sym}}}
\newcommand\geo{\gamma}
\newcommand\horofunction{h}
\newcommand\isometric{\cong}

Since we will be dealing with non-reversible metrics, we will discuss the
horofunction boundary in this setting. The development will be similar
to that in~\cite{walsh_stretch}.

Let $(X,\dist)$ be a non-reversible metric space, for example the Funk or
reverse-Funk geometry. We consider $X$ to be equipped with the topology
coming from the symmetrised metric $\symdist(x,y):=\dist(x,y)+\dist(y,x)$,
which for the two examples just mentioned is the Hilbert metric.

Associate to each point $z\in X$ the function $\distfn_z\colon X\to \R$,
\begin{equation*}
\distfn_z(x) := \dist(x,z)-\dist(b,z),
\end{equation*}
where $b\in X$ is an arbitrary fixed base-point. One can show that the map
$\distfn\colon X\to C(X),\, z\mapsto \distfn_z$ is injective and continuous.
Here, $C(X)$ is the space of continuous real-valued functions
on $X$, with the topology of uniform convergence on bounded sets of $\symdist$.
We define the \emph{horofunction boundary}\index{horofunction boundary}%
\index{boundary!horofunction} of $(X,\dist)$ to be
\begin{align*}
X(\infty):=\big(\closure\distfn(X)\big)\backslash\distfn(X),
\end{align*}
where $\closure$ is the topological closure operator. The elements of this
set are the \emph{horofunctions}\index{horofunction} of $(X,\dist)$.

Although this definition appears to depend on the choice of base-point,
one may verify that horofunction boundaries coming from different base-points
are homeomorphic, and that corresponding horofunctions differ by only an
additive constant.

Recall that a geodesic in a non-reversible metric space $(X,\dist)$ is 
a map $\gamma$ from a closed interval of $\R$ to $X$ such that
$\dist(\gamma(s),\gamma(t)) = t-s$, for all $s$  and $t$ in the domain
satisfying $s<t$. We make the following assumptions.
\renewcommand{\theenumi}{\Roman{enumi}}
\begin{enumerate}
\item
\label{assump1}
the metric $\symdist$ is proper, that is, its closed balls are compact;
\item
\label{assump2}
there exists a geodesic in $(X,\dist)$ between any two given points;
\item
\label{assump3}
for any point $x$ and sequence $x_n$ in $X$, we have $\dist(x_n,x)\to 0$
if and only if $\dist(x,x_n) \to 0$.
\end{enumerate}
Both the Funk and reverse-Funk metrics satisfy these assumptions.
Indeed, for~(\ref{assump2}), one may take the straight line segment between the
two points, parameterised appropriately.

Assumption~(\ref{assump1}) implies that uniform convergence on bounded sets
is equivalent to uniform convergence on compact sets.
Moreover, for elements of $\closure\distfn(X)$, the latter is equivalent
to pointwise convergence, by the Ascoli--Arzel\`a theorem, since elements
of this set are equi-Lipschitzian with respect to $\symdist$.
Again from the Ascoli--Arzel\`a theorem, the set $\closure\distfn(X)$
is compact. We call it the \emph{horofunction compactification}%
\index{horofunction compactification}\index{compactification!horofunction}.

Under assumptions~(\ref{assump1}),~(\ref{assump2}), and~(\ref{assump3}),
one may show that $\distfn$ is an embedding of $X$ into $C(X)$,
in other words, that it is a homeomorphism from $X$ to its image in $C(X)$.
From now on we identify $X$ with its image.

\subsection{Busemann points}

Define an \emph{almost-geodesic}\index{almost-geodesic} to be a path
$\gamma\colon T\to X$ such that $T\subset\Rplus$ contains $0$ and does not
contain $\sup T$, and such that, for any $\epsilon>0$,
\begin{equation*}
\big|\dist(\gamma(0),\gamma(s))+\dist(\gamma(s),\gamma(t))-t\big|<\epsilon,
\end{equation*}
for $s, t\in T$ large enough, with $s\le t$.
This definition has been modified from the original in~\cite{rieffel_group},
where $T$ was required to be an unbounded subset of $\Rplus$ containing $0$.
We must relax the condition that $T$ be unbounded because,
in the reverse-Funk metric, one may approach
the boundary along a path of bounded length.
See Figure~\ref{fig:finite_dist_to_boundary} for an illustration.
This modification is very minor.
\begin{figure}
\centering
\input{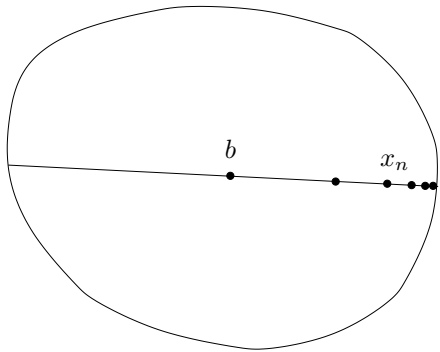}
\caption{Finite distance to the boundary.
$\funk(x_n,b)$ and $\rev(b,x_n)$ are bounded above, even though the sequence
$x_n$ approaches the boundary.}
\label{fig:finite_dist_to_boundary}
\end{figure}

Rieffel proved that every almost-geodesic converges. In his setting, it will
be to a horofunction, but in ours it may be to a point of $X$.
Every point of $X$ is trivially the limit of an almost-geodesic under
our definition. We say that a horofunction is a
\emph{Busemann point}\index{Busemann point}
if it is the limit of an almost-geodesic. Non-Busemann points of the
horofunction boundary have been found for normed spaces~\cite{walsh_normed}
and various finitely-generated groups with their word
metrics~\cite{walsh_artin,walsh_nilpotent,winweb_busemann}.

We define the \emph{detour cost}\index{detour cost}
for any two horofunctions $\xi$ and $\eta$ in $X(\infty)$ to be
\begin{align*}
H(\xi,\eta)
   &= \sup_{W\ni\xi} \inf_{x\in W\intersection X} \Big( d(b,x)+\eta(x) \Big),
\end{align*}
where the supremum is taken over all neighbourhoods $W$ of $\xi$ in
$X\cup X(\infty)$.
This concept appears in~\cite{AGW-m}. An equivalent definition is
\begin{align*}
H(\xi,\eta) &= \inf_{\gamma} \liminf_{t\to\sup T}
                 \Big( d(b,\gamma(t))+\eta(\gamma(t)) \Big),
\end{align*}
where the infimum is taken over all paths $\gamma\colon T\to X$ converging
to $\xi$.

The proof of the following proposition is easy.

\begin{proposition}
\label{prop:H_properies}
For all horofunctions $\xi$, $\eta$, and $\nu$,
\begin{enumerate}[(i)]
\item
\label{itema}
$H(\xi,\eta) \ge 0$;
\item
\label{itemb}
$H(\xi,\nu) \le H(\xi,\eta) + H(\eta,\nu)$.
\end{enumerate}
\end{proposition}

\begin{proposition}
\label{prop:zero_on_busemann}
Assume that $(X,d)$ satisfies
assumptions~(\ref{assump1}),~(\ref{assump2}), and~(\ref{assump3}),
and let $\xi$ be a horofunction.
Then, $H(\xi,\xi)=0$ if and only if $\xi$ is Busemann.
\end{proposition}
The proof of this proposition may be found in~\cite{walsh_stretch}.
There, it was assumed that $\dist(b,x_n)$ converges to infinity
for a sequence $x_n$ in $X$ whenever $\symdist(b,x_n)$ does.
Although this is not true for the reverse-Funk metric, it was only needed to
ensure that the almost-geodesic constructed is defined on all of $\Rplus$.
Since the definition of almost-geodesic has been modified in the current
setting, this is no longer necessary.

Propositions~\ref{prop:H_properies} and~\ref{prop:zero_on_busemann} together
say that, on the set of Busemann points,
the detour cost is a \emph{weak metric},
sometimes also called an \emph{extended pseudo quasi-metric}.

By symmetrising, one may obtain a genuine metric on the set of Busemann points
$X_B(\infty)$. For Busemann points $\xi$ and $\eta$ in $X_B(\infty)$, define
\begin{equation*}\label{eq:3.4}
\delta(\xi,\eta) := H(\xi,\eta)+H(\eta,\xi).
\end{equation*}
This construction first appeared in~\cite[Remark~5.2]{AGW-m}.
See~\cite{walsh_stretch} for a proof of the following proposition.
\begin{proposition}
\label{prop:delta_metric}
The function $\delta\colon X_B(\infty)\times X_B(\infty)\to [0,\infty]$
is a (possibly $\infty$-valued) metric.
\end{proposition}
We call $\delta$ the
\emph{detour metric}\index{detour metric}\index{metric!detour}.

Note that we can partition $X_B(\infty)$ into disjoint subsets such that
$\delta(\xi,\eta)$ is finite for each pair of Busemann points $\xi$ and $\eta$
lying in the same subset. We call these subsets the \emph{parts}
\index{parts of boundary} of the horofunction boundary of $(X,d)$,
and $\delta$ is a genuine metric on each one.

Consider an isometry $g$ from one metric space $(X,d)$ to another $(X',d')$.
We can extend $g$ continuously to the horofunction boundary $X(\infty)$ of $X$
as follows:
\begin{align}
\label{eqn:action_on_boundary}
g(\xi)(x') := \xi(g^{-1}(x'))-\xi(g^{-1}(b')),
\end{align}
for all $\xi\in X(\infty)$ and $x'\in X'$. Here $b'$ is the base-point of $X'$.

\subsection{Horoballs}
\newcommand\hausdorffdist{\delta_{\operatorname{Haus}}}
\newcommand\sublevel{\operatorname{slv}}

\newcommand\pul{\operatorname{Ls}}
\newcommand\pll{\operatorname{Li}}

To visualize horofunctions, it is often useful to consider their
horoballs\index{horoball},
that is, their sublevel sets. We denote the sublevel set of a function $f$
at height $\alpha$ by $\sublevel(f,\alpha):=\{x\in X\mid f(x)\le \alpha\}$.
Intuitively, a horoball looks like a large ball of the metric whose center
is at infinity.

To make this precise, we define a topology on the set of closed
subsets of $X$, called the Painlev\'e--Kuratowski topology%
\index{Painlev\'e--Kuratowski topology}\index{topology!Painlev\'e--Kuratowski}.
In this topology, a sequence of closed sets $(C_n)_{n\in\N}$
is said to converge to a closed set $C$ if the upper and lower closed limits
(also known as the limit superior and limit inferior) of
the sequence both equal $C$. These limits are defined to be, respectively,
\begin{align*}
\pul C_n &:= \bigcap_{n\ge 0} \closure \Big( \bigcup_{i>n} C_i \Big)
\qquad
\text{and} \\
\pll C_n &:=
   \bigcap\Big( \closure \bigcup_{i\ge 0} C_{n_i}
         \mid \text{$(n_i)_{i\in\N}$ is an increasing sequence in $\N$} \Big).
\end{align*}
An alternative characterisation of convergence is that $(C_n)_{n\in\N}$
converges to $C$ if and only if each of the following hold:
\begin{itemize}
\item
for each $x\in C$, there exists $x_n\in C_n$ for all $n$ large enough,
such that $(x_n)_n$ converges to $x$;
\item
if $(C_{n_k})_{k\in\N}$ is a subsequence of the sequence of sets and
$x_k\in C_{n_k}$ for each $k\in\N$, then convergence of $(x_k)_{k\in\N}$
to $x$ implies that $x\in C$.
\end{itemize}
See the book~\cite{beer_book} for more about this and other topologies
on sets of closed sets.

For each $z\in X$ and $\alpha\in\R$, we denote by
$B(z,\alpha):= \{x\in X \mid \dist(x,z)\le\alpha\}$ the right closed ball
about $z$ of radius $\alpha$.
In~\cite{karl_metz_nosk_horoballs}, it was shown that in a geodesic metric
space each horoball $\sublevel(\xi,\alpha)$ of a horofunction $\xi$ that is
the limit of a sequence $z_n$ takes the form $\pul B(z_n,\dist(b,z_n)+\alpha)$.
Here we improve this result by establishing a necessary and sufficient
condition for convergence to a horofunction in terms of convergence
of balls in the Painlev\'e--Kuratowski topology.

\begin{lemma}
\label{lem:small_downhill}
Assume $(X,d)$ satisfies~(\ref{assump1}), (\ref{assump2}), and (\ref{assump3}).
Let $x\in X$, and let $\xi\in X(\infty)$ be a horofunction.
Then, for all $\epsilon\ge 0$ small enough, there exists $y\in X$ such that
$\xi(x) - \xi(y) = d(x,y) = \epsilon$.
\end{lemma}

\begin{proof}
Choose a sequence $(z_n)$ in $X$ converging to $\xi$. Since no subsequence
of $(z_n)$ converges to $x$, we have that $(z_n)$ eventually remains outside
any sufficiently small open ball $\{z\in X\mid d(x,z)<\epsilon\}$ about $x$.
So, for $n$ large enough,
we may choose a point $y_n$ on a geodesic from $x$ to $z_n$ such that
$d(x,z_n)-d(y_n,z_n)=d(x,y_n)=\epsilon$. But some subsequence of $(y_n)$
converges to a point $y\in X$, and, taking limits, we get the result.
\end{proof}

\begin{lemma}
\label{lem:sublevels_continuous}
Assume $(X,d)$ satisfies~(\ref{assump1}), (\ref{assump2}), and (\ref{assump3}).
Let $\xi\in X(\infty)$ be a horofunction. Then, $\sublevel(\xi,\cdot)$ is
continuous in the Painlev\'e--Kuratowski topology.
\end{lemma}

\begin{proof}
Let $\alpha\in\R$, and let $(\alpha_n)$ be a sequence in $\R$ converging
to $\alpha$.

Let $(n_k)$ be an increasing sequence in $\N$, and let $(x_k)$ be a sequence
of points in $X$ converging to some point $x\in X$,
such that $x_k\in\sublevel(\xi,\alpha_{n_k})$ for all $k\in\N$. So,
\begin{align*}
\xi(x)
   = \lim_{k\to\infty} \xi(x_k)
   \le \lim_{k\to\infty} \alpha_{n_k}
   = \alpha.
\end{align*}
Therefore, $x\in\sublevel(\xi,\alpha)$.

Now let $y\in\sublevel(\xi,\alpha)$. For any $c\in\R$, we use the notation
$c^+:=\max(c,0)$. By Lemma~\ref{lem:small_downhill}, for $n$ large enough,
there exists $y_n\in X$ such that
\begin{align*}
\xi(y)-\xi(y_n)
   = d(y,y_n)
   = (\alpha-\alpha_n)^+.
\end{align*}
Observe that $(\alpha-\alpha_n)^+$ is a non-negative sequence converging to
zero. It follows that $(y_n)$ converges to $y$, and that
$y_n\in\sublevel(\xi,\alpha_n)$ for all $n$ large enough.
\end{proof}

The Painlev\'e--Kuratowski topology can be used to define a topology
on the space of lower-semicontinuous functions on $X$ as follows.
Recall that the epigraph of a function $f$ on $X$ is the set
$\{(x,\alpha)\in X\times\R\mid f(x)\le\alpha\}$. A sequence of
lower-semicontinuous functions is declared to be convergent in the
\emph{epigraph topology}\index{epigraph topology}\index{topology!epigraph}
if the associated epigraphs converge in the
Painlev\'e--Kuratowski topology on $X\times\R$. For proper metric spaces
the epigraph topology is identical to another topology called the
Attouch--Wets topology.

\begin{lemma}
\label{lem:function_converges}
Make assumption~(\ref{assump1}), which is that $(X,\symdist)$ is a proper
metric space.
Let $\xi$ be a real-valued lower-semicontinuous function on $X$, and let
$\xi_n$ be a sequence of such functions that is equi-Lipschitzian with
respect to $\symdist$. Then, $\xi_n$ converges to $\xi$
in the epigraph topology if and only if it converges to $\xi$ uniformly on
bounded sets.
\end{lemma}

\begin{proof}
This is a consequence of~\cite[Lemma~7.1.2]{beer_book}
and~\cite[Proposition~7.1.3]{beer_book}, and the fact that,
for proper metric spaces, Attouch--Wets
convergence is equivalent to epigraph convergence.
\end{proof}

We will also need the following result relating convergence of functions
to convergence of their sublevel sets. Recall that a proper metric space
is always separable.

\begin{proposition}[{\cite[Theorem~5.3.9]{beer_book}}]
\label{prop:function_convergence}
Let $f$ be a lower-semicontinuous function in a separable metric space,
and let $f_n$ be a sequence of such functions.
Then, $f_n$ converges to $f$ in the epigraph topology if and only if
there exists, for all $\alpha\in\R$, a sequence $\alpha_n$ in $\R$ converging
to $\alpha$ such that $\sublevel(f_n,\alpha_n)$ converges to
$\sublevel(f,\alpha)$ in the Painlev\'e--Kuratowski topology.
\end{proposition}

Now we can prove our result concerning the convergence of balls.

\begin{proposition}
\label{prop:convergence_of_balls}
Let $z_n$ be a sequence in $X$. Then, $z_n$ converges to
a point $\xi$ in the horofunction boundary if and only if,
for each $\alpha\in\R$, the sequence of balls $B(z_n,d(b,z_n)+\alpha)$
converges to $\sublevel(\xi,\alpha)$ in the Painlev\'e--Kuratowski topology.
\end{proposition}
\begin{proof}
Suppose that the balls converge as stated. Observe that
\begin{align*}
\sublevel(\distfn_{z},\alpha) = B(z,d(b,z)+\alpha),
\qquad\text{for all $z\in X$ and $\alpha\in\R$}.
\end{align*}
So, by Proposition~\ref{prop:function_convergence}, $\distfn_{z_n}$ converges
to $\xi$ in the epigraph topology, which is equivalent,
by Lemma~\ref{lem:function_converges}, to uniform convergence on bounded sets.

Now suppose that $z_n$ converges to $\xi$ in the horofunction boundary,
and let $\alpha\in\R$. Choose $\beta<\alpha$.
From Lemma~\ref{lem:function_converges}
and Proposition~\ref{prop:function_convergence},
we get that there exists a sequence $\beta_n$
in $\R$ converging to $\beta$ such that $B(z_n,d(b,z_n)+\beta_n)$ converges
to $\sublevel(\xi,\beta)$.
For $n$ large enough, $B(z_n,d(b,z_n)+\alpha)$ contains
$B(z_n,d(b,z_n)+\beta_n)$.
Therefore, $\pll B(z_n,d(b,z_n)+\alpha)$ contains $\sublevel(\xi,\beta)$.
Since this is true for all $\beta<\alpha$, and $\sublevel(\xi,\cdot)$
is continuous by Lemma~\ref{lem:sublevels_continuous}, we get
\begin{align*}
\pll B(z_n,d(b,z_n)+\alpha) \supset \sublevel(\xi,\alpha).
\end{align*}
The upper bound on the upper closed limit is proved in a similar manner.
\end{proof}

Figure~\ref{fig:converging_balls} illustrates the convergence of balls
to horoballs in the case of the Hilbert metric on the $2$-simplex.
\begin{figure}
\centering
\parbox{4.1cm}
   {\includegraphics[scale=0.8]{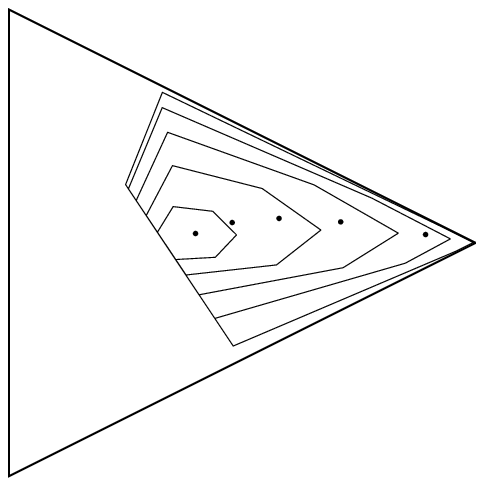}}
\qquad
\begin{minipage}{4.1cm}%
   \includegraphics[scale=0.8]{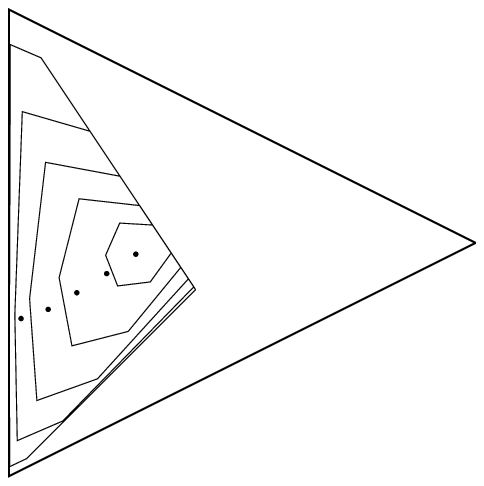}
\end{minipage}%
\caption{Sequences of balls converging to horoballs in the Hilbert
geometry.}
\label{fig:converging_balls}
\end{figure}

\section{The horofunction boundary of the Hilbert geometry}
\label{sec:hilbert_horoboundary}

In this section, we study the horofunction boundary of the Hilbert
metric on a bounded open convex set $D$.
As mentioned before, we will consider the horofunction boundaries of the
Funk and reverse-Funk geometries separately, and then combine the information
to get the boundary of the Hilbert geometry. Unless otherwise stated,
results come from~\cite{walsh_hilbert}.

\subsection{The horofunction boundary of the reverse-Funk geometry}

Observe that the function $\rev(\cdot,\cdot)$, defined on $D\times D$ extends
continuously to a function on $D\times\closure D$.
So, for all $z\in\closure D$, we may define on $D$ the function
$\hrev_{z}(\cdot):= \rev(\cdot,z)-\rev(b,z)$.
The following proposition shows that the horofunction compactification
of the reverse-Funk geometry is basically the same as the compactification
obtained by taking the closure of the domain in the usual topology.
We use $\partial D := \closure D \backslash D$ to denote the boundary of $D$
in the usual topology.

\begin{proposition}
\label{pro:reversehorofunctions}
Let $D$ be a bounded open convex set.
The set of horofunctions in the reverse-Funk geometry on $D$ is
$\revboundary:=\{\hrev_{x} \mid x\in\partial D \}$.
A sequence in $D$ converges to
$\hrev_{x}\in \revboundary$ if and only if it converges in the
usual sense to $x$.
\end{proposition}

Since straight line segments are geodesic in the reverse-Funk geometry,
it is clear that every horofunction in this geometry is Busemann.

The detour cost in this geometry was calculated
in~\cite{lemmens_walsh_polyhedral}.
Recall that a convex subset $E$ of a convex set $D$ is said to be an
\emph{extreme set}\index{extreme set}
if the endpoints of any line segment in $D$ are contained
in $E$ whenever any relative interior point of the line segment is.
The relative interiors of the extreme sets of a convex set $D$ partition $D$.

When we consider $\rev$ etc.~on convex sets other than $D$,
we use a subscript to specify the set.

\begin{proposition}[{\cite[Proposition~4.3]{lemmens_walsh_polyhedral}}]
\label{prop:reverse_detour}
Let $x$ and $y$ be in the usual boundary $\partial D$ of $D$.
Then, the detour cost in the reverse-Funk geometry is
\begin{equation*}
\detourrev(\hrev_{x},\hrev_{y}) = \rev_D(b,x) + \rev_F(x,y) - \rev_D(b,y),
\end{equation*}
if $y$ is in the smallest extreme set $F$ of $\closure D$ containing $x$,
and is infinity otherwise.
\end{proposition}
It follows immediately that the detour metric in this geometry is given by
\begin{equation*}
\deltarev(\hrev_{x},\hrev_{y}) = \hil_F(x,y),
\end{equation*}
when $x$ and $y$ are in the relative interior of the same extreme set $F$
of $\closure D$, and is infinity otherwise. Thus, the parts of the horofunction
boundary are just the relative interiors of the proper extreme sets of
$\closure D$.

\subsection{The horofunction boundary of the Funk geometry}

The horofunctions of the Funk geometry can be best understood by using
convex duality because they take a simpler form in the dual space.
Consider first the cone over $D$, in other words,
\begin{equation*}
C := \big\{ (p,1)\lambda \in \R^{n+1}
             \mid \text{$p\in D$ and $\lambda>0$}
     \big\}.
\end{equation*}
Here we are assuming that $D$ lives within $\R^n$.
We identify $D$ in the obvious way with a cross-section of this cone.
One may extend the function $\funk_D(\cdot,\cdot)$ from $D$ to $C$ in the
following way. For $x\in\R^{n+1}$ and $y\in C$, define the
\emph{gauge}\index{gauge}
\begin{equation*}
M(x/y;C) := \inf\big\{\lambda>0 \mid x\le_C \lambda y \big\}.
\end{equation*}
Here $\le_C$ is the partial ordering on $\R^{n+1}$ induced by $C$,
that is, $x\le_C y$ if $y-x\in\closure C$.
One can show \cite[page 29]{MR2953648}, 
that $\log M(x/y;C)=\funk_D(x,y)$, for all $x$ and $y$ in $D$.

\newcommand\dotprod[2]{\langle{#1},{#2}\rangle}
\newcommand\expf{j}

For each $z\in D$, the map $\distfn_z$ has the following extension to the whole
of $\mathbb{R}^{n+1}$:
\begin{equation*}
\distfn_z (x) :=
   \log \frac{M(x/z;C)}{M(b/z;C)},
\qquad\text{for all $x\in \mathbb{R}^{n+1}$}.
\end{equation*}

Define the (closed) dual cone of $C$:
\begin{align*}
C^*:=\{u\in\R^{n+1}\mid \text{$\dotprod{u}{x}\ge 0$ for all $x\in C$}\}.
\end{align*}
Here $\dotprod{\cdot}{\cdot}$ denotes the standard Euclidean inner product.
Observe that, by the Hahn--Banach separation theorem, $x\le_C y$
if and only if $\dotprod{u}{x}\le\dotprod{u}{y}$ for all $u\in C^*$.
It follows that
\begin{equation}
\label{eqn:M_as_ratio}
M(x/y;C) = \sup_{u\in C^*\backslash\{0\}}
              \frac{\langle u,x\rangle}{\langle u,y\rangle},
\qquad\text{for all $x\in\R^{n+1}$ and $y\in C$.}
\end{equation}
So, we see that $M(\cdot/z;C)$ is a convex function for fixed $z$.
Therefore, the same is true for the following function defined on $\R^{n+1}$:
\begin{equation*}
j_{C,z}(\cdot) := \frac{M(\cdot/z;C)}{M(b/z;C)} = \myexp\after\distfn_z(\cdot).
\end{equation*}

\newcommand\indicator{I}

Recall that the
\emph{Legendre--Fenchel transform}\index{Legendre--Fenchel transform}
of a convex function $f\colon \R^{n+1} \to \R\union\{\infty\}$ is the function
$f^*\colon \R^{n+1} \to \R\union\{\infty\}$ defined by
\begin{equation*}
f^*(y):= \sup_{x\in \R^{n+1}} \big( \dotprod{y}{x} - f(x) \big),
\qquad\text{for all $y\in \R^{n+1}$}.
\end{equation*}
The Legendre--Fenchel transform is a bijection from the set of
proper lower-semi\-continuous convex functions to itself, and is in fact a
homeomorphism in the Attouch--Wets topology.

Using~(\ref{eqn:M_as_ratio}), one may calculate without much difficulty the
Legendre--Fenchel transform of $j_{C,z}$.
We use $\indicator_E$ to denote the characteristic function of a set $E$,
that is, the function taking value $0$ on $E$ and $+\infty$ everywhere else.
For any open cone $U$ in $\R^{n+1}$ and any point $x\in U$, define
\begin{align*}
Z_{U,x} :=
   U^*\intersection\{u\in \R^{n+1} \mid \M{b}{x}{U} \dotprod{u}{x} \le 1 \}.
\end{align*}

\begin{proposition}
\label{prop:conjugates}
Let $C\subset \R^{n+1}$ be an open cone.
For all $z\in C$, we have that $\expf^*_{C,z}$ is the characteristic function
of the set $Z_{C,z}$
\end{proposition}

So the set $Z_{C,z}$ on which $\expf^*_{C,z}$ is zero is the intersection of
$C^*$ with a half-space bounded by a hyperplane on which $\dotprod{\cdot}{z}$
is constant. This constant is the largest possible such that $Z_{C,z}$ lies
within $\{u\in\R^{n+1}\mid\dotprod{u}{b}\le 1\}$.
This is illustrated in Figure~\ref{fig:zeroset}.

\begin{figure}
\centering
\input{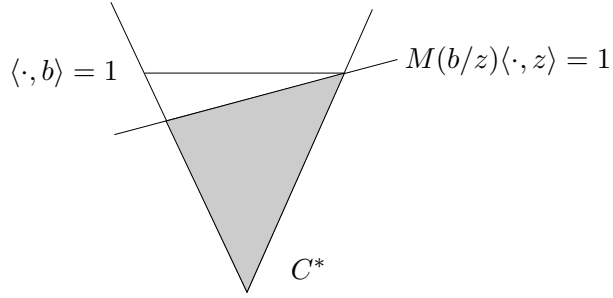}
\caption{The support of $\expf^*_{C,x}$}
\label{fig:zeroset}
\end{figure}

As $z$ approaches the boundary of $D$, the hyperplane becomes more and more
tilted. In the limit, the zero set will lie entirely within the (usual)
boundary of $C^*$.
We consider the limit of the zero sets in the Painlev\'e--Kuratowski topology.

\begin{proposition}
\label{prop:zero_set_converges}
A sequence $z_n$ in $D$ converges to a point in the horofunction boundary
of the Funk geometry if and only if $Z_{C,z_n}$ converges in the
Painlev\'e--Kuratowski topology.
\end{proposition}

We have the following simple description of the Busemann points.

\begin{proposition}
\label{prop:funk_busemanns}
A function $\xi\colon D\to\R$ is a Busemann point of the Funk geometry
if and only if it is the restriction to $D$ of a function of the form
$\log\indicator^*_{Z_{E,x}}$, where $E$ is an open convex cone
such that $E^*$ is a proper extreme set of $C^*$, and $x\in E$.
\end{proposition}

It is clear from this proposition that in some cases there may be horofunctions
that are not Busemann. This will happen for instance when there is a subset
$F^*$ of $C^*$ that is not extreme but is the limit of extreme sets.
In this case, the function $\log\indicator^*_{Z_{F,b}}$ will be,
by Proposition~\ref{prop:zero_set_converges}, a horofunction,
but will not be, by Proposition~\ref{prop:funk_busemanns}, a Busemann point.

%

The next result shows that this phenomenon is the only way that
non-Busemann points may arise.

\begin{proposition}
\label{prop:funk_nonbusemann_criterion}
All horofunctions of the Funk geometry on a domain $D$ are Busemann
if and only if the set of extreme sets of the polar of $D$ is closed
in the Painlev\'e--Kuratowski topology.
\end{proposition}

Recall that the polar $D^\circ$ of $D$ may be identified with a cross-section
of $C^*$.

The detour metric in the Funk geometry was calculated
in~\cite{lemmens_walsh_polyhedral}.

\begin{proposition}
\label{prop:funk_detour}
Let $\xi=\log\indicator^*_{Z_{E,x}}$ and $\eta=\log\indicator^*_{Z_{F,y}}$
be two Busemann points of the Funk geometry. The distance between them
in the detour metric is
\begin{equation*}
\deltafunk(\xi,\eta) = \hil_{E}(x,y),
\end{equation*}
if the extreme sets $E^*$ and $F^*$ of $C^*$ are equal,
and is infinity otherwise.
\end{proposition}

Thus, there is a part for each extreme set of the polar of $D$.

\subsection{The horofunction boundary of the Hilbert geometry}

It is evident from the expression of the Hilbert metric as the symmetrisation
of the Funk metric that every horofunction in the Hilbert geometry is the
sum of one in the Funk geometry and one in the reverse-Funk geometry.
The following proposition shows that there is a unique such decomposition of
each Hilbert horofunction.

\begin{proposition}
\label{hilbert_convergence}
A sequence converges to a point in the Hilbert-geometry horofunction boundary
if and only if it converges to a horofunction in both the Funk and reverse-Funk
geometries.
\end{proposition}
\begin{proof}
This follows from the proof of~\cite[Theorem~1.3]{walsh_hilbert}.
\end{proof}

Note that, combined with Proposition~\ref{pro:reversehorofunctions},
this implies that every sequence that converges to a Hilbert
horofunction also converges to a point in the usual boundary.
This generalises a result of~\cite{karl_foertsch}, where is was shown that
every Hilbert-geometry geodesic converges to such a point.

We also have the following description of the Busemann points.

\begin{theorem}
\label{hilbert_busemanns}
Let $h=r_x + f$ be a Hilbert horofunction written as the sum of a reverse-Funk
horofunction $r_x$ and a Funk horofunction $f$. Then, $h$ is a Busemann point
if and only if $f$ is.
\end{theorem}

Combining this with Proposition~\ref{prop:funk_nonbusemann_criterion}, we
get the following.

\begin{corollary}
All horofunctions of the Hilbert geometry on a domain $D$ are Busemann
if and only if the set of extreme sets of the polar of $D$ is closed
in the Painlev\'e--Kuratowski topology.
\end{corollary}

It remains to determine which combinations of Funk and reverse-Funk
horofunctions give Hilbert horofunctions.
The answer is given in the next theorem.
Recall that every point in $\partial C$ is a supporting functional of $C^*$,
and so defines an exposed face of $C^*$.

\begin{theorem}
\label{compatibility}
Let $r_x$ be a reverse-Funk horofunction,
and let $f:=\log\indicator^*_{Z_{E,z}}$
be a Busemann point of the Funk geometry.
Then, the function $r_x + f$ is a Busemann point
of the Hilbert geometry if and only if the extreme set $E^*$ of $C^*$
is contained in the exposed face of $C^*$ defined by $x$.
\end{theorem}

\begin{proof}
This is just a restatement of~\cite[Theorem~1.1]{walsh_hilbert}
using the description given above of the Funk Busemann-points in the dual space.
Indeed, by~\cite[Lemma~3.13]{walsh_hilbert}, the tangent cone $\tau(C,x)$
of $C$ at $x$ is dual to the exposed face of $C^*$ defined by $x$.
Also, by~\cite[Lemma~3.14]{walsh_hilbert}, an open cone $T$ can be obtained
from another $U$ by taking tangent cones if and only if $T^*$ is
an extreme set of $U^*$. The final ingredient is that the extreme sets
of an exposed face of a closed cone are exactly the extreme sets of the cone
that are contained within the exposed face.
\end{proof}

Using Propositions~\ref{prop:reverse_detour} and~\ref{prop:funk_detour},
one may find the detour metric of the Hilbert geometry.

\begin{theorem}[\cite{lemmens_walsh_polyhedral}]
\label{thm:hilbert_detour}
Let $\xi=r_p + \log\indicator^*_{Z_{E,x}}$ and
$\eta=r_q + \log\indicator^*_{Z_{F,y}}$
be two Busemann points of the Hilbert geometry. Then, the distance between them
in the detour metric of the Hilbert geometry is
\begin{equation}
\label{eqn:part_structure}
\deltahil(\xi,\eta) = \hil_G(p,q) + \hil_{E}(x,y),
\end{equation}
when $p$ and $q$ are in the relative interior of the same extreme set $G$
of $\closure D$, and the extreme sets $E^*$ and $F^*$ of $C^*$ are equal.
Otherwise, it is infinity.
\end{theorem}

So, we see that each part of the horofunction boundary of a Hilbert
geometry is the direct product of two lower-dimensional Hilbert geometries
with the $\ell_1$-sum distance.

\section{Isometries of the Hilbert metric}
\label{sec:hilbert_isometries}

De la Harpe~\cite{delaharpe} was the first to consider the isometry group
of the Hilbert geometry. Let $\mathbb{P}^n=\mathbb{R}^n\cup\mathbb{P}^{n-1}$
be real $n$-dimensional projective space, and suppose that $D$ is contained
within the open cell $\mathbb{R}^n$ inside $\mathbb{P}^n$.
Let $\mathrm{Coll}(D)$ be the set of collineations\index{collineation},
that is, elements of $\PGL(n+1,\mathbb{R})$ preserving $D$.
As de la Harpe observed, each element of $\mathrm{Coll}(D)$ is an isometry
since collineations preserve the cross-ratios.

\subsection{Simplicial geometries}
However, not every isometry is a collineation, as we will now see.
Think of the $n$-simplex as being a cross-section
of the positive cone $\R_+^{n+1}$. Hilbert's projective metric on the interior
of this cone is given by
\begin{align*}
\hildist(x,y)
   = \log\max_i \frac{x_i}{y_i}
    +\log\max_i \frac{y_i}{x_i}.
\end{align*}
Define $V$ to be the $n$-dimensional linear space $\R^{n+1}/\sim$,
where $x\sim y$ if
$x=y +\alpha (1,1,\ldots,1)$ for some $\alpha \in\R$, and equip
$V$ with the \emph{variation norm}:
\begin{equation*}
\|x\|_{\mathrm{var}} := \max_i x_i - \min_i x_i. 
\end{equation*}
We see that the Hilbert metric on the simplex is isometric to the normed
space $V$ via the map
\begin{align*}
\log\colon \interior\R_+^{n+1}\to\R^{n+1}, \,(x_i)_i\mapsto(\log x_i)_i
\end{align*}
that takes logarithms coordinate-wise.
This isometry was first observed by Nussbaum~\cite{nussbaum:hilbert} and
de la Harpe~\cite{delaharpe}. It was shown by Foertsch and
Karlsson~\cite{karl_foertsch} that simplices are the only Hilbert geometries
isometric to normed spaces.
\begin{theorem}
\label{thm:normed}
A finite-dimensional Hilbert geometry on a bounded open convex set is isometric
to a normed space if and only if the domain is a simplex.
\end{theorem}

Using this correspondence, and the Masur--Ulam theorem, the isometry group
of the simplical Hilbert geometry was determined for $n=2$
in~\cite{delaharpe} and for arbitrary $n$ in~\cite{lemmens_walsh_polyhedral}.
Let $\sigma_{n+1}$ be the group of coordinate
permutations on $V$, let $\rho\colon V\to V$ be the isometry given by
$\rho(x)=-x$ for $x\in V$, and identify the group of translations in $V$
with $\mathbb{R}^n$.
\begin{theorem}[\cite{francaviglia_martino_isometry,lemmens_walsh_polyhedral}]
\label{thm:1.2}
If $D$ is an open $n$-simplex with $n\geq 2$, then
\[
\mathrm{Coll}(D) \cong \mathbb{R}^n\rtimes \sigma_{n+1}
\quad\mbox{and}\quad 
\mathrm{Isom}(D) \cong \mathbb{R}^n\rtimes \Gamma_{n+1}, 
\]
where $\Gamma_{n+1} =  \sigma_{n+1}\times\langle\rho\rangle$.
\end{theorem}
The same result was obtained independently
in~\cite{francaviglia_martino_isometry}, although the authors there were
not aware that the geometry they were considering is isometric to a normed
space. Consequently, their proof is somewhat longer.

So, we see that in the case of the simplex the isometry group differs from
the collineation group. Indeed, the latter is a subgroup of index two of the
former. Let us look in more detail at the map $\rho$, which corresponds on
$\interior\R_+^{n+1}$ to the coordinate-wise reciprocal map
\begin{align}
\label{eqn:recip}
\hat\rho \colon \interior\Rplus^{n+1}\to\interior\Rplus^{n+1},
   \,(x_i)_i\mapsto\Big(\frac{1}{x_i}\Big)_i.
\end{align}
The projective action of this map is an isometry but not a collineation.
This is shown in Figure~\ref{fig:flip} for the case of $n=2$.

\begin{figure}
\centering
\parbox{4.1cm}
   {\includegraphics[scale=0.8]{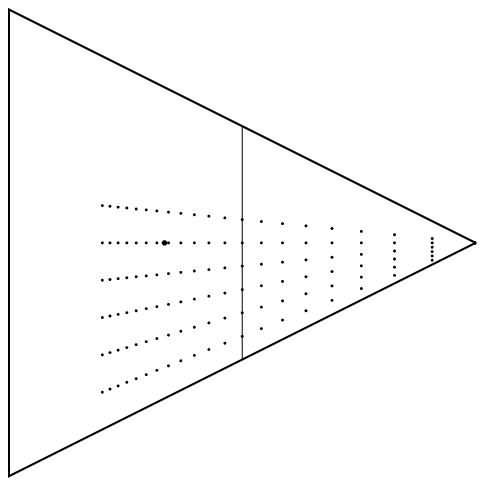}}
\qquad
\begin{minipage}{4.1cm}%
   \includegraphics[scale=0.8]{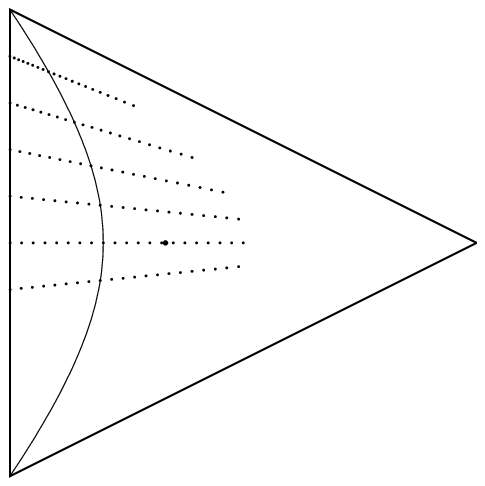}
\end{minipage}%
\caption{An isometry of the triangle that is not a collineation.
The solid straight line on the left is mapped to the curve on the right.} 
\label{fig:flip}
\end{figure}

As observed in~\cite{delaharpe}, the vertices are blown up to edges,
the edges are collapsed to vertices, and a straight line in the interior is
mapped to a straight line if and only if one of its ends approaches a vertex
of the triangle.

\subsection{Polyhedral geometries}

At the opposite extreme from the case of simplices, we have the following.
\begin{theorem}[\cite{delaharpe}]\label{thm:strictly_convex}
If the closure of $D$ is strictly convex, then
$\mathrm{Isom}(D) = \mathrm{Coll}(D)$.
\end{theorem}
It was also shown in the same paper that the only two-dimensional polyhedra
with non-collineation isometries are the triangles. This was generalised
to arbitrary dimension in the following theorem.
\begin{theorem}[\cite{lemmens_walsh_polyhedral}]
\label{thm:polyhedral}
If $(D,\hil)$ is a polyhedral Hilbert geometry, then
$\mathrm{Coll}(D)$ differs from $\mathrm{Isom}(D)$
if and only if $D$ is an open $n$-simplex, with $n\geq 2$.
\end{theorem}
The proof involves studying the action of an isometry on the parts of the
horofunction boundary, in particular, on the parts where one or other of the
$\ell_1$ factors is trivial, that is,
consists of a single point.
There is one such part for every vertex on the polyhedron, and one for every
facet. It is shown that every isometry either maps vertex parts to vertex
parts and facet parts to facet parts, or it interchanges them. Isometries of
the former type are shown to extend continuously to the Euclidean boundary
of the domain and to be collineations, whereas isometries of the latter type
are shown to only exist on simplices.

The above theorem verifies, for the case of polyhedral Hilbert geometries,
some conjectures of de la Harpe, namely that the Hilbert isometry group
$\isom(D)$ is a Lie group,
and that its identity component coincides with that of $\coll(D)$.
Observe that $\coll(D)$ is naturally a Lie group since it is a closed subgroup
of $\PGL(n+1,\R)$.
A consequence of these conjectures is
that $\isom(D)$ acts
transitively on $D$ if and only if $\coll(D)$ does, the latter happening
exactly when $D$ is the cross-section of an homogeneous cone. The only
polyhedra that occur as cross-sections of such cones are the simplices.
De la Harpe enumerates all possible cross-sections of homogeneous cones
up to dimension four~\cite{delaharpe}.

\subsection{The general case}

Dropping the polyhedral assumption, can one find other domains
with isometries that are not collineations? Consider again the
reciprocal map $\hat\rho$ defined in~(\ref{eqn:recip})
on the interior of the $n+1$-dimensional positive cone.
One way of seeing that it is an isometry of Hilbert's projective metric is
to observe that it is order-reversing, homogeneous of degree $-1$, and
involutive, meaning that it is its own inverse.

\newcommand\conemap{g}

The following proposition then implies that $\hat\rho$ is
gauge-reversing\index{gauge-reversing map}, that is,
\begin{align*}
\gauge{\Rplus^{n+1}}{\hat\rho(x)}{\hat\rho(y)} = \gauge{\Rplus^{n+1}}{y}{x},
\qquad\text{for all $x,y\in\interior\R_+^{n+1}$}.
\end{align*}
\begin{proposition}
[\cite{noll_schaffer_orders_gauge,nussbaum:hilbert}]
\label{prop:gauge-reversing}
Let $\conemap\colon  C\to C'$ be a bijection between two open cones in a
finite-dimensional vector space. Then, $\conemap$ is gauge-reversing if and
only if it is order-reversing and homogeneous of degree $-1$ and its inverse
is order-reversing.
\end{proposition}
It is clear that gauge-reversing bijections are isometries of Hilbert's
projective metric.

However, gauge-reversing maps do not just exist on the positive cone,
they exist on all
\emph{symmetric} cones\index{symmetric cone}\index{cone!symmetric},
of which the positive cone is an example.
Recall~\cite{faraut_koranyi} that a proper open cone $C$ in a
finite-dimensional real vector space is called \emph{symmetric}
if it is homogeneous, meaning that its group of linear automorphisms acts
transitively on it, and self dual, meaning that $C=C^\star$,
where $C^\star$ now denotes the \emph{open} dual.
One defines the \emph{characteristic function}\index{characteristic function}
$\phi$ on $C$ to be
\[
\phi(x) =\int_{C^\star} e^{-\langle y,x\rangle} \,dy,
\qquad\text{for all $x\in C$}.
\]
This map is homogeneous of degree minus the dimension of the cone,
which implies that
Vinberg's \emph{$*$-map}\index{Vinberg's $*$-map}\index{$*$-map},
\begin{align*}
C \to C^\star, \, x\mapsto x^* := -\nabla\log\phi(x),
\end{align*}
is homogeneous of degree $-1$.
It is known~\cite{kai_order_reversing} that on symmetric cones the
$*$-map is order-reversing.
In fact, it was shown in~\cite{kai_order_reversing} that this property of the
$*$-map characterises the symmetric cones among the homogeneous cones.
One may easily verify that the map $\hat\rho$ in~(\ref{eqn:recip})
is the $*$-map for the positive cone $\interior\R_+^{n+1}$.

\newcommand\herm{\operatorname{Herm}}
\newcommand\C{\mathbb{C}}
\newcommand\quaternions{\mathbb{H}}
\newcommand\octonions{\mathbb{O}}

The symmetric cones have been completely classified.
Each such cone is the product of one or more of the following irreducible
cones: the positive definite Hermitian matrices $\herm(n,E)$, with $n\ge 3$,
where the set of entries $E$ can be the reals $\R$, the complex numbers $\C$,
or the quaternions $\quaternions$; the  positive definite Hermitian $3\times 3$
matrices $\herm(3,\octonions)$ with octonion entries;
the \emph{Lorentz cone}\index{Lorentz cone}\index{cone!Lorentz}
(or ice-cream cone)\index{ice-cream cone}\index{cone!ice-cream}
\begin{align*}
\Lambda_n=\big\{(x_1,\ldots,x_n)\in \mathbb{R}^n
   \mid \mbox{$x_1>0$ and $x_1^2-x_2^2-\cdots -x_n^2>0$}\big\},
\end{align*}
for some $n\geq 2$.

For the reasons discussed above, the $*$-map on a symmetric cone is an
isometry of Hilbert's projective metric.
Its projective action is therefore an isometry
of the Hilbert geometry on a cross-section $D$ of the cone.
This isometry is not a collineation except when the symmetric cone
is a Lorentz cone.

It was conjectured in~\cite{lemmens_walsh_polyhedral}
that $\isom(D)$ and $\coll(D)$ differ if and only if the cone generated by $D$
is symmetric and not Lorentzian,
in which case the isometry group was conjectured to be generated by the
collineations and the isometry coming from the $*$-map.
This was known to be true for the cone of positive-definite Hermitian
matrices~\cite{molnar_thompson_isometries}, and has recently been established
for general symmetric cones~\cite{bosche}.

The conjecture has now been proved in
general~\cite{walsh_gauge}.
An important step in the proof is the following.

\begin{theorem}[\cite{walsh_gauge}]
\label{thm:symmetric}
An open convex cone in a finite-dimensional real vector space
admits a gauge-reversing map if and only if it is symmetric.
\end{theorem}

Establishing the homogeneity of such a cone admitting a gauge-reversing map
is not difficult. One uses Rademacher's theorem to get that the map is
differentiable almost everywhere. At each point of differentiability,
the negative of the differential is a linear automorphism of the cone.
This gives a large supply of automorphisms with which to prove homogeneity.

Proving the self-duality of the cone is more complicated. The first step
is to use arguments similar to those in the proof of homogeneity to show
there is a map $\conemap\colon C\to C$ that is gauge-reversing and involutive,
and fixes the base-point~$b$. One then wishes to define a positive-definite
bilinear form with respect to which $C$ is self-dual.

This is done by exploiting the fact that certain Funk horofunctions,
namely those associated to extreme rays of $C^*$, are the logarithms of linear
functionals. Each of these horofunctions is a singleton part and so gets
mapped by $\conemap$ to a singleton part of the reverse-Funk horofunction
boundary, that is, a reverse-Funk
horofunction of the form $r_x$, where $x$
is an extremal generator of $C$, in other words, a non-zero element of
an extremal ray of $\closure C$.

We now define
\begin{align*}
A(y,x) := \gauge{C}{x}{b} \myexp \after r_x \after \conemap(y)
       =  \gauge{C}{x}{\conemap(y)},
\end{align*}
for all $y\in C$ and extremal generators $x$ of $C$.
The equality holds because of the definition of $r_x$.
From what we have just seen, $\myexp \after r_x \after \conemap$ is the
restriction
to $C$ of a linear functional. This lets us extend the definition of $A$
to all $y$ in $V$ in such a way that it is linear in this variable.
Further work is then required to extend the definition to all $x$ in $V$
and verify that $A$ has the right properties.

One may also be interested in gauge-preserving
maps\index{gauge-preserving map}, that is, maps
$\conemap\colon C\to C$ satisfying
$\gauge{C}{\conemap(x)}{\conemap(y)} = \gauge{C}{x}{y}$,
for all $x$ and $y$ in $C$.
Such maps are also clearly isometries of Hilbert's projective metric.
The following proposition parallels Proposition~\ref{prop:gauge-reversing}.
\begin{proposition}
[\cite{noll_schaffer_orders_gauge,nussbaum:hilbert}]
\label{prop:gauge-preserving}
Let $\conemap\colon  C\to C'$ be a bijection between two open cones in a
finite-dimensional vector space. Then, $\conemap$ is gauge-preserving
if and only if it is order-preserving and homogeneous and its inverse is
order-preserving.
\end{proposition}

These maps are well understood.

\begin{theorem}[\cite{noll_schaffer_orders_gauge,rothaus_order_isomorphisms}]
\label{thm:linear}
Let $C$ and $C'$ be open cones in a finite-dimensional vector space, and let
$\conemap\colon C\to C'$ be gauge-preserving bijection.
Then, $\conemap$ is the restriction to $C$ of a linear isomorphism.
\end{theorem}

Here is a simple proof using horofunctions.

\begin{proof}
Choose the base-points $b\in C$ and $b'\in C'$ such that $\conemap(b)=b'$.
Since $\conemap$ is gauge-preserving, it preserves the Funk metric,
and therefore maps Funk horofunctions to Funk horofunctions according to
(\ref{eqn:action_on_boundary}).
In particular, singleton horofunctions are mapped to singleton horofunctions.
Recall that there is a singleton horofunction of the Funk geometry for
each extremal ray of the closed dual cone; in fact the horofunction is just
the logarithm of a suitably normalised element of the ray.
So for every extremal generator $x$ of $C^*$, there is an extremal generator
$x'$ of $(C')^*$ satisfying $\dotprod{x}{y} = \dotprod{x'}{\conemap(y)}$,
for all $y\in C$.
However, by choosing the right number of extremal generators of $C^*$
we obtain a basis for the dual space, that is, a linear coordinate system.
It follows that $\conemap$ is the restriction to $C$ of a linear map.
\end{proof}

So the projective action of every gauge-preserving map is a collineation.
The conjecture referred to above now follows from 
combining Theorems~\ref{thm:symmetric} and~\ref{thm:linear} with the following
fact proved in~\cite[Theorem 1.3]{walsh_gauge}:
every isometry of a Hilbert geometry arises as the projective
action of either a gauge-preserving map or a gauge-reversing map.

\bibliographystyle{abbrv}
\bibliography{bibliohilbert.bib}


\printindex

\end{document}